\documentclass[12pt,leqno]{article}

\usepackage{amsthm}
\usepackage{amsfonts}
\usepackage{amssymb}
\usepackage{verbatim}
\usepackage[all]{xy}
\usepackage{amsmath, amscd, amsthm}
\usepackage{longtable}
\usepackage{amscd}
\usepackage{mathrsfs}
\usepackage{graphicx}
\usepackage{latexsym}

\newtheorem{Lemma}{\sc Lemma}  
\newtheorem{Theorem}{\sc Theorem}  
\newtheorem{Proposition}{\sc Proposition}  
\newtheorem{Corollary}{\sc Corollary}  
  
\newtheorem*{Claim}{\sc Claim}  
\theoremstyle{remark}
\newtheorem{Remark}{\sc Remark}

\usepackage{latexsym}

\newcommand{\vp}{\varphi}
\newcommand{\ld}{\ldots}

\newcommand{\ve}{\varepsilon}

\renewcommand{\th}{^{\textrm{th}}}
\newcommand{\st}{^{\textrm{st}}}

\newcommand{\ad}{\textrm{ad\:}}
\newcommand{\gr}{\textrm{gr\:}}
\newcommand{\id}{\textrm{id\,}}

\newcommand{\A}{\mathcal{A}}

\newcommand{\Z}{{\mathbb Z}}

\newcommand{\fib}{\mathrm{Fib}}
\newcommand{\supp}{\mathrm{Supp}\,}
\newcommand{\Lp}[1]{\mathrm{Lp}(#1)}

\theoremstyle{definition}
  
\theoremstyle{plain}
\begin{document}
\title{Growth of subalgebras and
subideals in free Lie algebras}
\author{Yuri Bahturin\thanks{Supported by NSERC Discovery grant \# 227060-09}
\\Department of Mathematics and Statistics\\Memorial University of Newfoundland\\St. John's, NL, A1C5S7
\and Alexander Olshanskii\thanks{Supported by NSF grant
DMS 1161294 and RFFR grant 11-01-00945}
\\Department of Mathematics\\
1326 Stevenson Center\\
Vanderbilt University\\
Nashville, TN 37240
}
\date{}
\maketitle
\begin{abstract}
We investigate subalgebras in free Lie algebras, the main tool being relative growth and cogrowth functions. Our study reveals drastic differences in the behavior of proper finitely generated subalgebras and nonzero subideals. For instance, the \textit{growth} of a proper finitely generated subalgebra $H$ of a free Lie algebra $L$, with respect to any fixed free basis $X$, is exponentially small compared to the growth of the whole of $L$. Quite opposite, the \textit{cogrowth} of any nonzero subideal $S$ is exponentially small compared to the growth of $L$.
 \end{abstract}  

\section{Introduction}\label{sINT}

According to the classical Shirshov - Witt theorem, every subalgebra of a free Lie algebra is free. There are a number of ways to draw distinction between different subalgebras. In this paper we are interested in the (relative) growth and cogrowth of subalgebras. 

Any free basis $X$ of a free Lie algebra $L=L(X)$ defines an increasing filtration 
\begin{displaymath}
L^{(1)}\subset L^{(2)}\subset\ldots\subset L^{(n)}\subset\ldots,
\end{displaymath}
 where $\dim L^{(i)}<\infty$, if $X$ is finite. The (relative) growth function of a subspace $H\subset L$ is an integer valued function of an integral argument whose $n\th$ value equals $\dim (H\cap L^{(n)})$. The $n\th$ value of the cogrowth function equals $\dim L^{(n)}/(H\cap L^{(n)})$.

Let us recall that if $\# X = m \ge 2$ then the growth of the whole of $L=L(X)$ is exponential. This easily follows from the classical Witt's formula \cite[Theorem 3.1.3]{B} for the dimension of the $n\th$ homogeneous component $L_n$ of $L$
\begin{equation}\label{en1}
\dim L_n=\frac{1}{n}\sum_{d|n}\mu(d)m^{\frac{n}{d}}\sim \frac{1}{n}m^n.
\end{equation}

Here $\mu$ is the M\"obius function and the equivalence means that the ratio of the left hand side to the right hand side tends to 1 as $n\to\infty$. The exponentiality holds for the growths and the cogrowths of the most of subalgebras studied in this paper but a more precise comparison shows that the degree of exponentiality can be rather different.

To be precise, we call a function $f(n)$ of natural argument with nonnegative  values \textit{exponential}, if the limit of the function $\sqrt[n]{f(n)}$ exists and is greater than 1. We call $\displaystyle\lim_{n\to\infty}\sqrt[n]{f(n)}$ the (exponential) \textit{base} of $f(n)$. 

Let us say that a (growth) function $f_1(n)$ is \textit{exponentially negligible} when compared to another (growth) function $f_2(n)$ if the ratio $\dfrac{f_2(n)}{f_1(n)}$ grows exponentially. We say that the growth or the cogrowth of a subalgebra $H$ of a free Lie algebra $L=L(X)$, with a fixed free basis $X$, is exponentially negligible if the respective relative growth (cogrowth) function for $H$ is exponentially negligible when compared with the growth function for $L$. The exponential base of the relative growth function for $H$ in $L$ with respect to a fixed free basis $X$ is denoted by $\beta_H$.

\medskip

In the study of the growth and cogrowth of subalgebras of free Lie algebras,  the best exlored is the case of ideals.  The cogrowth of an ideal $ I $ of a free Lie algebra $ L $ is actually the growth of the factor-algebra $ L/ I $. This notion is one of the main in the theory of infinite-dimensional Lie algebras. Many classes of algebras are defined in terms of their growth. For a general monograph on the topic see \cite{KL}. In the theory of simple Lie algebras, confirming a conjecture by V. Kac \cite{VK} (see also his monograph \cite{KM}), O. Mathieu classified simple Lie algebras of polynomial growth \cite{OM}. The growth of solvable Lie algebras and algebras satisfying polynomial identities also was studied by a number of authors. For a survey of this area see \cite{VP}. 

\medskip

Note that in the case of ideals, the cogrowth of a nonzero ideal $I$ of a free Lie algebra $L=L(X)$ of rank $\ge 2$ is exponentially negligible when compared with the growth of $L$. This easily follows from the results of \cite[Subsection 2.2]{BO} where the respective claim has been established for free associative algebras. Indeed, consider the free associative algebra $A=A(X)$ as the universal enveloping algebra of $L$ and the two-sided ideal $J$ of $A$ generated by $I$. Then $L/I\subset A/J$ by Poincar\'e - Birkhoff - Witt Theorem and the growth of $A/J$ by \cite[Proposition 8]{BO} is negligible. Since by Witt's formula (\ref{en1}) the exponent of the growth of $L$ is the same as that of $A$, it follows that the growth of $L/I$ is negligible, as well.

\medskip

Much less explored is the case of subalgebras which are not ideals. In this paper we start with general finite-dimensional subalgebras. 

\begin{Theorem}\label{t1I} 
The (relative) growth function $g_H(n)$ of any  proper finitely generated subalgebra $H$ of a free Lie algebra $L$ of finite rank $m$ is exponentially negligible.
\end{Theorem}

In the case of subalgebras which are \textit {subideals} the situaion is entirely opposite.

\medskip
 
Let us recall that  a subalgebra $M$ of a Lie algebra $L$ is named an $\ell$-\textit {subideal}, $\ell=1,2,\ldots$, if there is a decreasing sequence

 \begin{displaymath}
L \triangleright L_1 \triangleright \ldots \triangleright L_{\ell-1}\triangleright M,
\end{displaymath}
 where each term is an ideal in the preceding term. 

The general theory of infinite-dimensional Lie algebras with a sizeable portion of material devoted to subideals is exposed in a book \cite{AS}; however, our present topic, the \textit{growth and cogrowth of subideals}, is not covered there.

\begin{Theorem}\label{t2I}
The cogrowth of any nonzero subideal of a free Lie algebra of rank $m\ge 2$ is exponentially negligible.
\end{Theorem}

The result of Theorem \ref{t2I} contrasts dramatically with the situation in the Group Theory where the cogrowth of a nontrivial subnormal subgroup of the free group $F_r$, $r\ge 2$, can be equivalent to the growth of $F_r$.

When compared to the proof of the same result for the ideals, as given earlier, the proof of this result required much more sophistication and even the development of a new piece of technique. This technique is presented in Section \ref{sCGW}. One of the immediate applications of the technique is the following result, which provides essential details to our claim of exponential negligibility of  finitely generated subalgebras in Theorem \ref{t1I}.  

\begin{Theorem}\label{t1aI} Let $L$ be a free Lie algebra of rank $m$, with a fixed  basis $X$. Then the following are true.
\begin{enumerate}
\item[\emph{(i)}] The relative growth function $g_H(n)$ of a proper nonabelian subalgebra $H$ of $L$, with respect to $X$, is exponential.
\item[\emph{(ii)}] If $H$ is finitely generated  nonabelian subalgebra of $L$ then the exponential base $\beta_H$ of $g_H$ is an algebraic integer. If $H$ is proper then $\beta_H<m$.
\item[\emph{(iii)}] For any real $m_0\in[1,m]$ there is a subalgebra $H$ of $L$ such that $\beta_H=m_0$. One can choose $H$ as a retract of an ideal of codimension 1 in $L$.
\item[\emph{(iv)}] For any real $m_0\in[1,m]$ there is a 2-subideal $S$ of $L$ such that the exponential base $\beta_{L/S}$ of the relative cogrowth function $g_{L/S}$ equals $m_0$. One can choose $S$ as an ideal in an ideal $M$ of codimension 1 in $L$, with free factor algebra $M/S$.
\end{enumerate}
\end{Theorem}

We do not know if for any $m_0\in[1,m]$ there is an ideal $S$ with $\beta_{L/S}=m_0$.

\medskip

To further emphasize the potential of the combinatorial approach to the questions of the growth\!/\!cogrowth, in Section \ref{sE}, as an example, we calculate the cogrowth function for the minimal 2-subideal containing one of the free generators of a free Lie algebra of rank 2.

Theorem \ref{t1I} gives a quantitative evaluation of the ``size'' of a finitely generated subalgebra $H$ when compared to the ``size'' of the whole of the free Lie algebra $L$. An example showing how the ``quantitative'' Theorems \ref{t1I} and \ref{t2I}, can be applied to obtain purely ``qualitative'' consequences, is as follows. 

\begin{Corollary}\label{c2I} No proper finitely generated subalgebra $K$ of
a free Lie algebra $L$ can contain a nonzero subideal $H$ of $L$.
\end{Corollary}

\proof We may assume that $L$ has a finite rank $m\ge 2$.
By Theorem \ref{t1I}, $K$ is negligible. If $H\subset K$, then also $H$ is negligible. However, this contradicts Theorem \ref{t2I}.
\endproof 

\medskip

Apart from the sharp difference in the character of growth/cogrowth, there are other properties where finitely generated subalgebras and subideals differ drastically. Let us recall that given a subalgebra $H$ of a Lie algebra $G$, the \textit{idealizer} of $H$ in $G$ is the unique maximal subalgebra $K$ containing $H$ as an ideal. Clearly, any proper subideal is different from its idealizer in any Lie algebra. At the same time, an easy consequence of our Theorem \ref{t1I} is the following, probably known fact: \textit{any nonzero finitely generated subalgebra of a nonabelian free Lie algebra is self-idealizing, that is, equal to its idealizer}. Using this property one could derive Corollary \ref{c2I} using a more traditional argument, like Theorem \ref{t4I} below. This result, for which we were unable to find a source in the literature, constitutes a natural counterpart to M. Hall's theorem about finitely generated subgroups in the free groups \cite{H}.

\begin{Theorem}\label{t4I}
Let $H$ be a finitely generated subalgebra of a free Lie algebra $L$ of finite rank. Then a free basis $B_0$ of $H$ can be complemented to produce a free basis $B$ of a subalgebra $C$ of finite codimension in $L$. In other words, if $K$ is the subalgebra generated by $B\setminus B_0$ then $C\cong H\ast K$, the free product of $H$ and $K$. 

One can always choose $B\setminus B_0$, hence $K$, homogeneous. If $H$ is itself homogeneous, one can choose $B$ so that $C$ contains the $t^\mathrm{th}$ term of the lower central series $L^t$, for some $t$.
\end{Theorem}

From this theorem we easily obtain the following consequence.

\begin{Corollary}\label{c4I}
Every proper finitely generated subalgebra $H$ in a free Lie algebra $L$ of finite rank is a free factor in strictly greater subalgebra of $L$; in addition, if homogeneous, $H$ is a free factor of a \emph{subideal} of $L$.
\end{Corollary}

A comparison with Group Theory reveals the following differences. First, in Hall's Theorem \cite{H}, saying that a finitely generated subgroup $H$ of a free group $F_r$ of rank $r$ is a free factor in a greater subgroup of $F_r$, one must stipulate that $H$ is of \textit{infinite index} in $F_r$. Second, in the case of Lie algebras, if $H$ is proper nonzero, the free factor $K$ appearing in Theorem \ref{t4I} can never be chosen finitely generated; whereas in the case of Groups, Hall's Theorem asserts that the complementary free factor $K$ is \textit{always finitely generated.}

When dealing with the cogrowth of subideals,  the following situation comes to mind quite naturally. If a subideal $ S $ contained a nonzero ideal $ I $ then the cogrowth of $ S $ would be exponentially negligible, without any further argument.  The result that follows shows that this such situation is highly unlikely.

\begin{Theorem}\label{t0601} Let $J$ be a proper ideal of a free Lie algebra $L$ and $S$ a finite subset of $J$. Then the ideal closure $I$ of $S$ in $J$ does not contain nonzero ideals of $L$.
\end{Theorem}

To state the corollaries of this theorem, we will use the following notation. Given a subset $S$ of a Lie algebra $L$, we denote by $\mathrm{id}_L S$, the ideal of $L$ generated by $S$. This is the minimal ideal of $L$ that contains $S$. We also say that $\mathrm{id}_L S$ is an \textit{ideal closure} of $S$ in $L$. If $\ell$  is a natural number, then an $\ell$-subideal of $L$ generated by $S$ is the subalgebra $\mathrm{id}^\ell_L S$ of $L$ which is defined by induction as follows. We set $\mathrm{id}^1_L S=\mathrm{id}_L S$ and if $\ell>1$ then define $\mathrm{id}^\ell_L S$ as the ideal of $\mathrm{id}^{\ell-1}_L S$ generated by $S$. An easy inductive argument shows that $\mathrm{id}^\ell_L S$ is contained in every $\ell$-subideal of $L$ containing $S$. We also say that $\mathrm{id}^\ell_L S$ is an \textit{$\ell$-subideal closure} of $S$ in $L$.

\begin{Corollary}\label{t3I1}Let $L$ be a free Lie algebra, $J$ a proper ideal of $L$, $S$ a finite subset of $J$. Then for no $\ell\ge1$ the $\ell$-subideal closure $\mathrm{id}_J^{\ell}S$ contains a nonzero $\ell$-subideal of $L$. 
\end{Corollary}
 
\begin{Corollary}\label{t3I}
Let $L$ be a free Lie algebra and $S$ a finite subset of $L$ whose ideal closure in $L$ is different from $L$. Then for any $\ell\ge 1$,  the $(\ell+1)\st$ subideal closure of $S$ in $L$ does not contain any nonzero $\ell$-subideal of L.	
\end{Corollary}

The restriction $\mathrm{id}_L(S)\ne L$ in the statement of the above result is necessary because, otherwise, for every $\ell=1,2,\ldots$, the $\ell$-subideal closure of $S$  equals $L$ and the claim is not valid.

\medskip
We conclude this section with few remarks about the differences of the situation with the growth and cogrowth in the case of algebras as opposed to the case of groups. We will also give some hints as for the methods we have used or developed. 

\medskip

Two observations make the situation in the case of algebras so different from the case of groups (and semigroups!). In the case of algebras, we measure the growth using the \textit{degree} of the elements while in the case of groups we deal with the \textit{length}. So our first observation is that the degree of the sum of two elements never exceeds the degree of the summands, while the length of the product does not need to obey this rule.  As for the second, the reader probably already noticed that in the proof of Corollary \ref{c2I} we have used a simple general purely ``algebraic'' property: \textit{The sum of the (relative) growth and cogrowth functions of a subalgebra always equals to the growth function of the whole algebra}. Again, no analogue of this claim holds valid in the case of groups. As a result, dealing with the growth and cogrowth functions is by necessity more delicate in the case of groups.

Still, to prove our results about the growth of subalgebras in the free Lie algebras here, it was necessary to suggest at least two pieces of technique we think new. 

One of them, suggested in Section \ref{sCS}, deals with the following issue. Alhough we define the growth functions for subalgebras in the free Lie algebra $L=L(X)$ using a finite fixed free basis $X$, in our proofs we need to switch from this fixed basis to other bases, which could be by necessity infinite. To preserve growth functions, we need to develop an approach to \textit{the growth with respect to infinite  graded bases} of free Lie algebras. This is done in Section \ref{sCS} where certain analytic finiteness conditions are suggested that allow us to estimate the cogrowth when we use induction to pass to infinitely generated ideals and subideals and at the same time have the degrees of elements unchanged. The second piece of technique is produced in Lemma \ref{lInvDer} where we examine what can be naturally called \textit{shifting derivations} in associative and Lie algebras.

\section{Preliminaries}\label{sP}

In this section we give necessary definitions concerning the relative growth of subspaces in algebras. We also recall some results about the connection between associative and Lie algebras, bases of free Lie algebras and their subalgebras.

\subsection{Relative growth of  subspaces in finitely filtered spaces}\label{ssRG}
Let $V$ be a vector space over a field $F$ with the filtration
 
\begin{equation}\label{eF}
\alpha:\quad \{0\}\subset V^{(0)}\subset V^{(1)}\subset\ldots\subset V^{(n)}\subset\ldots,\qquad V=\bigcup_{n=0}^\infty V^{(n)},  
\end{equation}

We say that $\alpha$ is a \textit{finite ascending filtration}
if all subspaces $V^{(n)}$ are finite-dimensional. The\textit{ growth function} of $\alpha$, $g_\alpha$, is then given by $g_\alpha (n)=\dim V^{(n)} $. The\textit{ graded growth function} of $\alpha$, $d_\alpha$,  is given by $d_\alpha(n)=\dim V^{(n)}/ V^{(n-1)}$ (assuming $V^{(-1)}=\{ 0\}$). The filtration $\alpha$ also defines a degree function $\deg_\alpha$ on the set of nonzero elements of $L$ given by $\deg_\alpha(a)=n$ provided that $a\in  V^{(n)}\setminus V^{(n-1)}$. 

If $W$ is subspace then the \textit{subfiltration} $\alpha\cap W$ is a filtration formed by the subspaces $W^{(n)}=W\cap V^{(n)}$. The factor-space $V/W$ acquires a \textit{factor-filtration} $\alpha/W$ given by  the subspaces $(V/W)^{(n)}=(V^{(n)}+W)/W\cong V^{(n)}/(W\cap V^{(n)})$. 

In each of these cases the filtrations of the respective spaces are finite. Once $\alpha$ and $g_\alpha$ are fixed, we can define the \textit{growth function relative to $\alpha$}: $g_{\alpha\cap W}(n)=\dim W^{(n)}$, and \textit{cogrowth function relative to $\alpha$}: $g_{\alpha/W}(n)=\dim (V/W)^{(n)}$. If the filtrations are fixed than we denote the respective growth functions as $g_V$, $g_W$ and $g_{V/W}$.

In each of these cases we have also \textit{relative graded functions} $d_V$, $d_W$ and $d_{V/W}$. Say, $d_{V/W}(n)=g_{V/W}(n)-g_{V/W}(n-1)$. 

It often so happens that the space $V$ is $\Z$-graded: $V=V_0\oplus V_1\oplus V_2\oplus\cdots\oplus V_n\oplus\cdots$, and for each term of the filtration $\alpha$ one has $V^{(n)}=V_0\oplus V_1\oplus V_2\oplus\cdots\oplus V_n$, for all $n=1, 2,\ldots$. The nonzero elements of $V_n$ are called \textit{homogeneous of degree} $n$ and $V_n$ the \textit{homogeneous subspace of the grading of degree} $n$. If $v\in V$ then $v$ can be uniquely written as the sum of some elements $v_n\in V_n$. We call $v_n$ a homogeneous components of $v$ of degree $n$. If $n$ is the greatest such that $v_n\ne 0$ then we call $v_n$ the \textit{leading part} of $u$ and write $v_n=\mathrm{Lp}(u)$. 

A subspace $U\subset V$ is called \textit{graded} if $U$ contains all homogeneous components of its elements. Equivalently, $U=(U\cap V_0)\oplus (U\cap V_1)\oplus (U\cap V_2)\oplus\cdots\oplus (U\cap V_n)\oplus\cdots$. If $W$ is not graded then the linear span of the leading parts of nonzero elements of $W$ is a graded subspace of $V$, which is denoted by $\gr W$. For each natural $n$, the homogeneous component $(\gr W)_n$ of this space is the linear span of the leading parts of all elements of degree $n$ in $W$.

An easy but useful technical result is the following.

\begin{Lemma}\label{lGR} For the relative growth functions, one has 
\begin{enumerate}
\item[$\mathrm{(a)}$] $g_W=g_{\emph{\gr} W}$;
\item[$\mathrm{(b)}$] $g_{V/W}=g_{V/\emph{\gr} W}$.
\end{enumerate}
\end{Lemma}

\proof
 (a) It follows from the definition of $(\gr W)_n$ that the mapping
$w+W^{(n-1)}\mapsto \mathrm{Lp}(w)$ for $w\in W^{(n)}\setminus W^{(n-1)}$ and $w\mapsto 0$ for $w\in W^{(n-1)}$ is a well- defined linear isomorphism between $W^{(n)}/W^{(n-1)}$ and $(\gr W)_n$.
Hence
  
\begin{displaymath}
g_W(n)=\dim W^{(n)}=\sum_{i=0}^n \dim  W^{(i)}/W^{(i-1)}=\sum_{i=0}^n \dim\:(\gr W)_i = g_{\gr W}(n).
\end{displaymath}

(b) Follows from (a) and  the equality  $\dim V^{(n)}/W^{(n)}=\dim V^{(n)}-g_W(n)$.
\endproof

\subsection{Linear and free bases of subalgebras in free Lie algebras}\label{ssRLFB}

As usual, given an alphabet $X$, we denote by $W(X)$ the monoid of all words in $X$, including the empty word $1$. In this section  we will need to count the words in the subsets of  $W(X)$, where $X$ is  not necessarily finite. If we do not impose any conditions, the number of words of the same length can be infinite. To cope with this, we consider the alphabets which are \textit{finitely graded}, that is, $X=\bigsqcup_{i=1}^{\infty}{X_n}$ with each  set $X_n$ finite. Now in addition to the \textit{lengths} of words, we can speak about the \textit{degrees} of words. We define the degree of a word $w$ in $W(X)$ by induction on its length if we set  $\deg (1)=0$,  $\deg (x)=n$ for $x\in X_n$ and if $w=x_1\cdots x_k$ then  $\deg (w)=\deg (x_1)+\cdots+\deg (x_k)$. Let $W(X)_n$ be the set of all words of degree $n$. Now each number $d_X(n)=\# W(X)_n$ is finite. We have $W(X)=\bigsqcup_{n=0}^\infty W(X)_n$ where  $W(X)_n$ is the set of all words of degree $n$ in $X$. 

In its turn, the free associative algebra $A(X)$, which is a vector space with basis $W(X)$, acquires a $\Z$-grading $A(X)=\bigoplus_{n=0}^\infty A(X)_n$, where $A(X)_n$ is the linear span of all words in $W(X)_n$.  Now in addition to the grading, $A(X)$  have a natural \textit{degree} filtration
\begin{displaymath}
A(X)^{(0)}\subset A(X)^{(1)}\subset\ldots\subset A(X)^{(n)}\subset\ldots
\end{displaymath}
where $A(X)^{(n)}= A(X)_0\oplus A(X)_1\oplus\ldots\oplus A(X)_n$, for each $n=0,1,2,\ldots$.

As mentioned earlier, every element $f\in A(X)$ acquires a degree if we set $\deg (f) =n$ in the case where $f\in A(X)^{(n)}$ but $f\not\in A(X)^{(n-1)}$. In this case $f=\sum_{i=0}^n f_i$, where $f_i\in A(X)_i$, $f_n\neq 0$; $f_n=\Lp{f}$ is the \textit{leading part} of $f$.

In what follows we are going to study the growth and cogrowth of certain subalgebras $H$ of a free Lie algebra $L=L(X)$. We view $L$ as a graded subspace of $A=A(X)$ which is generated by $X$ with respect to the bracket operation $[a,b]=ab-ba$. The cardinality $\# X$ of $X$ is often called the \textit{rank} of any of the system appearing in this text: $W(X)$, $A(X)$ or $L(X)$.
Being a graded subspace of $A$,  $L$ becomes $\Z$-graded: $L=\bigoplus_{n=1}^\infty L_n$ where $L_n=L\cap A_n$. Similarly, we have a degree filtration of $L$ as follows:
\begin{equation}\label{en2}
L^{(1)}\subset L^{(2)}\subset\ldots\subset L^{(n)}\subset\ldots
\end{equation}
where $L^{(n)}=L\cap A^{(n)}$ or 
$L^{(n)}= L_1\oplus\ldots\oplus L_n$, for each $n=1,2,\ldots$. Every element $u\in L$ has degree which is the same whether we use the degree filtration of $A$ or $L$. When we apply the notions of subsection \ref{ssRG} to the free Lie algebra, we view  (\ref{en2}) as the ``base'' filtration $\alpha$ of the formula (\ref{eF}) and then we know what is the degree, the leading part of an element of $L$ and, given a subalgebra $H\subset L$, what is the associated graded subalgebra $\gr H$. This is indeed a subalgebra because the commutator of two leading parts is ether zero or the leading part of the commmutator of two elements of $H$.
 
Using Lemma \ref{lGR}, one can see that the growth and the cogrowth of $H$ and $\gr H$ are the same.

In the next lemmas we consider the elements of the free Lie algebra $L(X)$ called the \textit{commutators in $X$} of certain length. The commutators of length 1 are just the elements of $X$. If the commutators of all lengths less than $n>1$ have been defined,  then the commutators of length $n$ are all the elements $[c,d]$ where $c$ is a commutator of length $k$ and $d$ a commutator of length $n-k$. 

Since we always view $L(X)$ as a subalgebra in the free associative algebra $A(X)$, recalling $[c,d]=cd-dc$, one can write each commutator as a linear combination of associative words, each of which has the same degree. Thus each commutator acquires a uniquely defined degree and $\deg ([c,d])=\deg (c)+\deg (d)$.

If we forget all brackets and commas on a commutator $c$, then resulting is an associative word $w$, we write $c=[w]$, called the \textit{associative support of} $c$. We can also say that any commutator can be obtained by replacing brackets (and there inherent commas) on an associative word. Replacing brackets in a certain way on certain associative words may produce a linear basis of the free Lie algebra $L(X)$. Let us describe a process due to A. I. Shirshov \cite{S}. 

Suppose we have a linear ordering $\le$ on the set $X$ of free generators of the free Lie algebra $L(X)$. We extends this ordering to the lexicographical ordering of the free monoid $W(X)$. Under this ordering,  given two words $u$ and $v$,  then we say that $u\ge v$ if $u$ is a prefix of $v$. If not, we say that $u <v$ if $u=wxu'$, $v=wyv'$ and $x<y$ where $x,y\in X$ and $w, u, v, u', v'\in W (X)$. 

If $c$ is a commutator then the \textit{leading word} $\bar c$ of $c$ is defined as the greatest word nontrivially entering the expansion of $c$ as a linear combination of the elements of $W(X)$. Let us say that $w$ is a \textit{Lyndon - Shirshov} or simply \textit{LS-word} if whenever $w=uv$, with $u$ and $v$ nonempty, then $w>vu$. (Notice that in this case we must also have $w>v$, that is, every LS-word is greater than its proper suffix.)

Let us call a commutator $c$ an \textit{LS-commutator} if $c=[w]$ where $w$ is an LS-word and the following two conditions hold. 

\begin{enumerate}
\item[{\rm(a)}] If $c=[c_1,c_2]$ then each $c_i$ is an LS-commutator with LS-support $w_i$, $i=1,2$, and $w_1>w_2$;
\item[{\rm(b)}] If  $c$ is as in (a) and $c_1=[c_1',c_1'']$ where $w_1''$ is the associative support of $c_1''$ then  $w_1''\le w_2$.
\end{enumerate}

In the remainder of this section, for the convenience of the reader, we state several results of A.I. Shirshov, which are now conveniently available in the collection \cite{SW}. They can also be found in the books \cite{B} and \cite{BK}.

\begin{Lemma}\label{lnS} On each LS-word $w$ one can replace brackets in a unique way so that the resulting commutator $c=[w]$ is an LS-commutator (such that ${\bar c}=w$). The set of all LS-commutators is a linear basis of $L(X)$. 
\end{Lemma}

\begin{Lemma}\label{lnS1} Any associative word $v\ne 1$, which is not necessarily an LS-word, can be uniquely written as the product of LS-words $v=u_1u_2\cdots u_s$ so that $1\ne u_1\le u_2\le\ldots\le u_s$.
\end{Lemma}

  To list a consequence of this fact, of importance to us, we introduce a convenient notation $[c_1,c_2,\ldots,c_k]$, called the \textit{left-normed commutator} of $c_1,c_2,\ldots,c_k$. This is defined by induction, starting with just $c_1$, if $k=1$. If $k>1$ then one sets $[c_1,c_2,\ldots,c_{k-1},c_k]=[[c_1,c_2,\ldots,c_{k-1}],c_k]$.

\begin{Corollary}\label{cnLS} If $c$ is an LS-commutator and $w={\bar c}=yv$ for some $y\in X$ and $w\in W(X)$ then $y$ is the maximal letter among all letters involved in the expression of $w$ and $c=[y,[u_1],[u_2],\ldots[u_s]]$, $u_1\le u_2\le\ldots\le u_s$, where each $[u_i]$ is the (unique) LS-commutator with the associative support $u_i$, $i=1,2,\ldots,s$.
\end{Corollary}

A standard technique in the theory of free Lie algebras is the \textit{elimination} of one or more elements of a free basis (sometimes called Lazard elimination \cite{NB}).

\begin{Lemma}\label{pnC1} Let $L(X)$ be a free Lie algebra with free basis $X=Y\bigsqcup\{ z\}$, where $Y\neq\emptyset$. Let $J=\mathrm{id}_L Y$ be the ideal of $L$ generated by $Y$. Let $B$ be the set of all left normed commutators $[y,z,\ldots,z]$, where $y\in Y$. Then $J$ is a free Lie algebra with free basis $B$.
\end{Lemma}
 
\section{Growth of finitely generated subalgebras of free Lie algebras}\label{sCGFGS}

As before, $L=L(X)$ stands for a free Lie algebra over a field $F$ with a free basis $X$. This section is devoted to the proof of Theorem \ref{t1I}.

Let us call a subset $S$ of a free Lie algebra $L$ \textit{irreducible} if no leading part of any of its elements belongs to the subalgebra generated by the leading parts of the remaining elements of $S$. In his proof of the theorem on the freeness of subalgebras of free Lie algebras \cite{S}, Shirshov shows that any subalgebra $M$ can be generated by an irreducible set $S$. Then he shows that any irreducible set is \textit{independent}, that is, may serve as a free basis for a subalgebra it generates.  We can summarize and slightly complement
 this as follows.

\begin{Lemma}\label{pnC0} Let $ S $ be a subset of a free Lie algebra $ L $,  $ S ' $ the set of leading parts of the elements in $ S $.  If $ S $ is irreducible then $ S ' $ is independent.  If $ S ' $ is independent then $S $ is also independent. Any subalgebra of a free Lie algebra can be generated by an irreducible set.
\end{Lemma}

\begin{proof} We only need to explain the second claim.  However,  it is true in any algebra which is free in a variety of algebras and endowed with a grading with respect to a free basis, that any nontrivial relation between the elements of a subset $S$ entails a nontrivial relation between the elements of $S'$.
\end{proof}

\begin{Lemma}\label{pnC2} Let $M$ be a subalgebra of a free Lie algebra $L$, $M'=\gr M$ the associated graded subalgebra. Let $S$ be an irreducible free basis of $M$ and $S'=\{ \Lp{s}\,|\,s\in S\}$ the set of leading parts of the elements in $S$. Then the map $\vp:S\to S'$ given by $\vp(s)=\Lp s$ extends to a degree preserving isomorphism $\overline{\vp}:M\to M'$. 
\end{Lemma}

\begin{proof} According to Lemma \ref{pnC0},  $S'$ is an independent set,  hence a free basis of a subalgebra $ N $. As a result, $\overline{\vp}$ is an isomorphism from $ M $ to $ N $.  We only need to show that $ N=M' $. Since $ S ' $ and $ M'$ are homogeneous,  it is sufficient to establish $\Lp {\overline{\vp}(u)}= \Lp {u} $, for any $ u\in M\setminus\{ 0\}$. To do this,  let us express $ u $ as a Lie polynomial $ u=f (s_1,\ldots, s_m)$, in $s_1,\ldots,s_m \in S$, and write $ f=f_1+f_2 $, where $ f_1 $ is a linear combination of monomials of the highest degree if we count each variable $s_i$ with the same degree as $\deg s_i $ in $ L $.  Then 
\begin{eqnarray*}
{\overline{\vp}(u) }&=& \Lp{\overline {\vp}(f_1(s_1,\ldots, s_m))+ \overline {\vp} (f_2(s_1,\ldots, s_m ))}\\&=&\Lp{f_1(s'_1, \ldots, s'_m )+f_2( s'_1, \ldots, s'_m )}=f_1( s'_1, \ldots, s'_m )=\Lp{u}, 
\end{eqnarray*}as needed. 
\end{proof}

Now we proceed to the proof of Theorem \ref{t1I}.

\proof  

According to Lemmas \ref{lGR} and \ref{pnC2}, given a finitely generated subalgebra $M$ of $L$, the growth and cogrowth of a finitely generated homogeneous subalgebra $M'=\gr M$ is the same as the growth and cogrowth of $M$. If $M$ is proper then the same is true for $M'$. Therefore, in proving our theorem we may restrict ourselves to the case where $M$ is a nonzero graded subalgebra. This entails, $m\ge 2$.

Let $\{ z_1,\ldots,z_k\}$ be a free homogeneous basis of $M$. Since $M$ is proper in $L$, the elements of degree 1 among $\{ z_1,\ldots,z_k\}$ form a basis $Z$ of $M\cap L_1$, which is proper in $L_1$. Let us complement $Z$ to a linear basis $Y$ of $L_1$. Then we will obtain a new free basis $Y$ of $L$, which properly includes $Z$.  As a result, the main degree filtration (\ref{en2}) does not change and so the relative growth functions of subalgebras do not change. The elements of degree 1 of this set are a proper part $Z$ of a free basis $Y$. Hence there is an element of the new free basis, say $y$, which is not an element of $Z$. 

As a result, without loss of generality, we may assume from the very beginning that the elements of degree 1 in the free homogeneous basis $\{ z_1,\ldots,z_k\}$ of $M$ are a part of the fixed free basis $\{ x_1,\ldots,x_m\}$ of $L$ and $x_1$ is not in $M$. Let us consider $L$ as a subalgebra in the free associative algebra $A=A(x_1,\ldots,x_m)$. Clearly,  $z_1,\ldots,z_k$ are the elements of the subalgebra $B$ of $A$ which is generated by finitely many monomials $u_1,\ldots,u_t$ (all the  monomials used to write the generators $z_1,\ldots,z_k$). By our assumption, we know that none of these monomials equals $x_1$. Moreover, we can actually assume that none of $u_1,\ldots,u_t$ is a power of $x_1$. Indeed, no such monomials can appear while writing the elements of the free Lie algebra $L(x_1,\ldots,x_m)$ in the free associative algebra $\A(x_1,\ldots,x_m)$. It remains to show that the exponent of the growth of $B$ is less than $m$. 

Let us choose an integer $d$ which is at least the double maximum of the degrees of all $u_1,\ldots,u_t$ with respect to $x_1,\ldots,x_m$. Since any of these monomials contains as a factor a letter different from $x_1$, no product can contain a subword $x^d$. It is known (see, e.g., \cite[Lemma 8]{BO} that for every nonempty word $w$ in $x_1,\ldots,x_m$  there exist $C,\ve>0$ such that the number of words of length $n$ which do not have $w$ as a subword is bounded by $C(m-\varepsilon)^n$.
\endproof

\begin{Remark}\label{rassoc0} There are some consequences of Theorem \ref{t1I} that are well-known (see e.g. \cite[Chapter 3]{B}). For instance, because the growth of a subspace of finite codimension is equivalent to the growth of the whole space, we conclude  that \textit{ in a free Lie algebra the proper nonzero subalgebras of finite codimension cannot be finitely generated}.  Another consequence, already mentioned in Introduction, is the following: \textit{ in a free Lie algebra any nonzero finitely generated subalgebra is self-idealizing}. 
\end{Remark}

\begin{Remark}\label{rassoc1} Theorem \ref{t1I} fails in the case of associative algebras. A simple example of a proper finitely generated subalgebra whose growth is not exponentially negligible is the subalgebra  $B$ generated by all monomials of degree 2 and 3 in the free associative algebra $A$ of rank $m\ge 1$. In this case $\dim A/B=m+1$, following because any number $n>1$ can be written as $n=2k+3l$, where $k$ and $l$ are non-negative integers.   

However, even if we assume $A/B$ infinite-dimensional, we still can have examples of not exponentially negligible finitely generated subalgebras. For instance, one can proceed as follows. Let $\alpha$ be the standard filtration on $A$ of rank $m\ge 1$ and $B$ a subalgebra of $A$ generated by all monomials of degree 2.  We obviously have $\dim A/B=\infty$, but at the same time the value of the growth function for the filtration $\beta=\alpha\cap B$ is $g_\beta(n)=\dfrac{m^{n+2}-1}{m^2-1}$, if $n$ is even. For $\alpha$ we have $g_\alpha(n)=\dfrac{m^{n+1}-1}{m-1}$. Clearly, $B$ is \textit{not} exponentially negligible.
\end{Remark}

\section{Some properties of words in infinite alphabets}\label{sCGW}

Infinite alphabets naturally appear when one applies a standard technique of elimination (see Lemma \ref{pnC1}). Although it is not our goal to formally generalize certain results to the case of free Lie algebras of infinite rank, the logic of the proofs makes it necessary to consider such algebras, as an auxiliary tool.

In what follows, we will be considering only the alphabets with at \textit{least two letters}. Let $k_n=\# X_n$ be the number of letters of degree $n$. In our analysis, we will be imposing on $X$ some conditions, as follows.

\bigskip
\textbf{Condition} $\mathrm{G}$: \textit{Either $k_2=k_3=...=0$ or, for every $i$, if  $k_i>0$ then $k_{i+1}>0$.}
\smallskip

To formulate another condition we first introduce a function of one real variable $\zeta$ ($z$ is a positive parameter): 
\begin{equation}\label{ezeta}
F(\zeta)=F^X_z(\zeta)=\sum_{i=1}^{\infty}\frac {k_i}{(z-\zeta)^i}.
\end{equation}
Its domain is the set of real values of $\zeta$ for which the series on the right hand side converges.

Recall (Subsection \ref{ssRLFB}) that $d_X(n)$ stands for the number of words of degree $n$ in the free monoid $W(X)$. A simple fact is the following.

\begin{Lemma} \label{bezslov1} 
If $F(0)\le 1$ then $d_X(n)\le z^n$, for all natural values of $n$.
\end{Lemma}

\proof Induction on $n=0,1,\ldots$. We have $d_X(0)=1$ and 
for $n\ge 1$, the number of words of degree $n$ with the last letter
from $X_i$ equals $d_X(n-i)k_i$. Applying induction and our hypothesis, we obtain:
\begin{displaymath}d_X(n)=\sum_{i=1}^n k_id_X(n-i)\le \sum_{i=1}^n k_i z^{n-i}\le z^{n}\sum_{i=1}^{\infty} \frac{k_i}{z^{i}}=z^nF(0)\le z^n.\end{displaymath}
\endproof

\begin{Lemma}\label{rSUP}
Let for some real $z>1$ the series $\displaystyle\sum_{i=1}^{\infty}\frac{k_i}{z^i}$ converge
to a number $\alpha>1$. Denote by $\delta$ the greatest common divisor of all integers in the set $I=\{ i\,|\, k_i\neq 0\}$. Then there is a positive constant $c$ such that for every sufficiently
large $n$ divisible by $\delta$, we have $d_X(n)>cz^n.$
\end{Lemma} 

\begin{proof}
Since the series converges to $\alpha>1$, there is $N$ such that
$\displaystyle\sum_{i=1}^{N}\frac{k_i}{z^i}>1$. 
It is well known that there is an integer $M$ such that
every $n>M$ divisible by $\delta$ is a linear combination of the numbers from $I$
with non-negative integral coefficients. Without loss of generality, let us assume that
both $M$ and $N$ are  divisible by $\delta$.
It follows that there exists a word of degree $n$ over the graded alphabet $X$,
and so there is a small real number $c>0$ such that $d_X(n) > cz^n$ for
every $n=\delta t$ belonging to the segment $[M, M+N],$ where $t$ is a positive integer.

The assertion of the lemma will be proved by induction on $t$, where $n=\delta t$, with base
$t=M/\delta$ guaranteed above by the choice of $c$. Moreover, for the inductive step,
we may assume that $n>M+N$ and so $t\ge (M+N)/\delta$.  Now we have
\begin{equation}\label{eSUP}
d_X(n)=\sum_{i=1}^{n}k_id_X(n-i)\ge\sum_{i=1}^{N}k_id_X(n-i)
\end{equation}
The right-hand side of (\ref{eSUP}) can be rewritten as $\displaystyle\sum_{i=1}^{N/d}k_{\delta i}d_X(\delta t-\delta i)$
 because obviously $d_X(j)=0$ if $j$ is not a multiple of $\delta$.
Note that $t>t-i\ge t-N/\delta>M/\delta$ for $1\le i\le N/\delta$, and so by the inductive hypothesis, 
$\displaystyle d_X(\delta t-\delta i)>cz^{\delta t-\delta i}=c\frac{z^n}{z^{\delta i}}$. Taking (\ref{eSUP}) into account, we obtain
$\displaystyle d_X(n)>cz^n\sum_{i=1}^{N/\delta}\frac{k_{\delta i}}{z^{\delta i}}$. Again, the sum on the right-hand side is equal
to $\displaystyle\sum_{i=1}^N \frac{k_{i}}{z^{i}}$, that is, greater than $1$ by the choice of $N$. Hence
$d_X(n)>cz^n$ as required.
\end{proof}

In the next result, we denote by $g_A(n)$ the growth function of $A(X)$ with respect to the degree filtration defined by the graded set $X$.

\begin{Lemma}\label{ldXn}
The function $g_A(n)$
is superexponential if the series F(0) diverges for every positive $z$. Otherwise,
provided that $\# X >1$, the growth of $g_A(n)$ is exponential. The base of the
exponent of the growth can be determined as $z_0$ in the following.
\begin{equation}\label{edXn}
\lim_{n\to\infty}\sqrt[n]{g_X(n)}=z_0=\inf \{z>1 \mid \sum_{i=1}^{\infty}\frac{k_i}{z^i}\le 1\}.
\end{equation}
Moreover, $d_X(n)\le z^n$ for any $n\ge 0.$
\end{Lemma}

\begin{proof} By Lemma \ref{bezslov1}, $d_X(n)\le z^n$ for any $z>z_0$, whence $g_X(n)\le c_1 z^n$ for a positive
constant $c_1$. By Lemma \ref{rSUP}, $d_X(n)> c(z')^n$ for any sufficiently large $n$ divisible by $\delta$
if $z'<z_0$. Therefore $g_X(n)>c_2(z')^n$ for a positive $c_2$ and every large enough $n$.
To finish the proof of formula (\ref{edXn}), we should extract the $n\th$ roots from the  inequalities   obtained for $g_X(n)$ and pass to the limit, using that $z$ and $z'$ can be chosen arbitrary close to $z_0$.
The second statement is proved in Lemma \ref{bezslov1}.
\end{proof}

\smallskip

To further proceed we need to impose one more condition on the set $X$.
 
\textbf{Condition} $\mathrm{W}_z\:(z$ a real number $>1)$: \textit{$F(\zeta)$ is defined in a neighborhood of $0$ and $F(0)\le 1$.}

\medskip

Our main Lemma in this section deals with the following sets. Given a nonempty word $u\in W(X)$, let $N(u)$ be the set of all words in $W(X)$ that do not include $u$ as a subword. We also set $N(u)_n=N(u)\cap W(X)_n$ and define $f_u(n)=\# N(u)_n$, that is, $f_u(n)$ is the number of all words of degree $n$ without subword $u$.

\begin{Lemma} \label{bezslov} Let $X$ satisfy Conditions $\mathrm{G}$ and $\mathrm{W}_z$.  Then for every nonempty word $u\in W(X)$ there exist positive constants $C$ and $\ve$, such that 
\begin{displaymath}
f_u(n)<C(z-\varepsilon)^n.
\end{displaymath}
\end{Lemma}

\proof With $u$ fixed, let us write $f(n)=f_u(n)$. If $u$ is a subword of $u'$ then, clearly, $N_n(u)\subset N_n(u')$.  Since $X$ has at least two letters, it is easy to include $u$ as a subword in a word $u'$ such that no proper prefix of  $u'$ is a suffix of $u'$. Therefore we may assume from the very beginning that the word $u$ itself enjoys this property.

It follows from Condition $\mathrm{W}_z$, that $F(\zeta)$ is continuous in a neighbourhood of $0$. Let us choose a small positive $\varepsilon\in (0;1)$ such that $F(\varepsilon)< 1+\dfrac{1}{z^\ell}$, where $\ell=\deg (u)$. Clearly one can choose $C$ so that $f(n)< C(z-\varepsilon)^n$ for every $n=1,\dots, \ell-1$. To show that the same inequality also holds for all $n\ge \ell$, we first apply the argument of Lemma \ref{bezslov1}. Namely, we use 
\begin{equation}\label{eNU}
N(u)_n\subset \bigsqcup_{i=1}^{n}N(u)_{n-i}X_{i},
\end{equation}
to produce an upper bound 
\begin{displaymath}
f(n)\le\sum_{i=1}^{n} f(n-i)k_{i}.
\end{displaymath}

Let us note that the union on the right hand side of (\ref{eNU}) contains some words with subword $u$. For example, such are the words in $N(u)_{n-\ell}u$. Indeed, if $u=u'x$ with $x\in X_s$, and $v\in N_{n-\ell}(u)$  then $vu=(vu')x\in N_{n-s}(u)X_s$.  This follows because, considering $v$ has no subword $u$, if $vu'$ has such a subword, then $ v=v' p$, $u=pq$ and $u'=qr$, for some words $v',p,q,r$, $q$ nonempty. Then a nonempty word $q$ would be a suffix and a prefix of $u$ at the same time, which contradicts our assumption about $u$.  Since the number of  the words in  $N(u)_{n-\ell}u$ is $f(n-\ell)$, the stronger inequality holds:
\begin{equation}\label{fn}
f(n)\le-f(n-\ell)+\sum_{i=1}^n k_if(n-i)
\end{equation}

Now according to our Conditon G  we either have G1: $k_2=k_3=\dots=0$ or G2: for every $i$, if  $k_i>0$ then $k_{i+1}>0$. Let us handle these cases separately.

In the case G1,  all letters are contained in $X_1$ and $f(n-1)\le f(n-\ell)z^{\ell-1}$ by Lemma \ref{bezslov1} since  in this case, every word of degree $n-1$ is a product  of a word of degree $n-\ell$ and a word of degree $\ell-1$.
Therefore we obtain from (\ref{fn}):
\begin{equation}\label{e478}
f(n)\le -f(n-1)z^{1-\ell}+\sum_{i=1}^n k_if(n-i)\le\sum_{i=1}^n k'_if(n-i),
\end{equation}
where $k'_1=k_1-z^{1-l}\ge 0$ and $k'_j=k_j$ for $j\ne 1$. Note that in this case, 
\begin{equation}\label{e479}
\frac{k'_1}{z-\varepsilon}=\frac{k_1}{z-\varepsilon}-\frac{1}{z^{\ell-1}(z-\varepsilon)}\le
\dfrac{k_1}{z-\varepsilon}-\dfrac{1}{z^{\ell}}.
\end{equation}

In the case G2, since the last letter $x$ of $u$ is in $X_s$, we have $k_s>0$. Also $\ell=\deg (u)=\deg (u'x)\ge \deg (x) =s$. As a result, we have $k_\ell>0$ and so we can write
\begin{equation}\label{e482}
f(n)\le -f(n-\ell)+\sum_{i=1}^n k_if(n-i)\le \sum_{i=1}^n k'_if(n-i),\end{equation}
where $k'_\ell=k_\ell-1\ge 0$ and $k'_j=k_j$ for $j\ne \ell.$ 
In this case we have
\begin{equation}\label{e483}
\dfrac{k'_\ell}{(z-\varepsilon)^\ell}\le \dfrac{k_\ell}{(z-\varepsilon)^{\ell}}-\dfrac{1}{z^{\ell}}.
\end{equation}  

Now we can make conclusions which are the same both for G1 and G2. Since the coefficients $k'_i$ are non-negative, we conclude from the inductive hypothesis and equations (\ref{e478}) and (\ref{e482}) that
\begin{equation}\label{fn1}
f(n)\le\sum_{i=1}^n k'_if(n-i)\le \sum_{i=1}^n Ck'_i(z-\varepsilon)^{n-i}=C(z-\varepsilon)^n\sum_{i=1}^n \frac{k_i'}{(z-\varepsilon)^{i}}
\end{equation}
Hence, both in G1 and G2,  it  follows from Equation (\ref{fn1}), the definition of the coefficients $k'_i$ and equations (\ref{e479}) and (\ref{e483})that
\begin{eqnarray*}
f(n)&\le& C(z-\varepsilon)^n\left(-\frac{1}{z^{\ell}}+\sum_{i=1}^n \frac{k_i}{(z-\varepsilon)^{i}}\right)\\&\le& C(z-\varepsilon)^n \left(-\frac{1}{z^{\ell}}+F(\varepsilon)\right)\\&<& C(z-\varepsilon)^n,
\end{eqnarray*}
because $F(\varepsilon)<1+\dfrac{1}{z^\ell}.$ The proof is now complete. \endproof

\section{Cogrowth of Subideals}\label{sCS}

This section contains the proof of Theorem \ref{t2I}. We will be using the notions and results about free associative and Lie algebras given in Subsection \ref{ssRLFB}.

\begin{Lemma}\label{ldegsub} Let $L=L(X)$ be a free Lie algebra with free linearly ordered basis $X$. Let $w$ be an LS-word, $x\in X$ the maximal letter involved in $w$ and $c=[w]$ the (unique) LS-commutator with associative support $w$. If  $x^\ell$ is a subword of $w$ for some $\ell>0$, then $c$ is an element of the $\ell$-subideal closure $\mathrm{id}_L^\ell x$ of $L$ generated by $x$. 
\end{Lemma}

\proof  By definition of the lexicographical order, it follows that $w=x^\ell w'$. According to Corollary \ref{cnLS}, $c=[x,[u_1],[u_2],\ldots,[u_s]]$, where $u_1\le u_2\le\ldots\le u_s$. If $\ell=1$ then $c$ is an element of the ideal generated by $x$. So we can proceed by induction on $\ell$ with basis $\ell=1$. Assume $\ell>1$. Since $[u_1]$ and $c$ are LS-commutators, $u_1$ cannot  be a power of $x$, hence $u_1$ has a proper prefix $x^{\ell-1}$. By the property of the lexicographical order, $x^{\ell-1}$ is also a prefix in each $u_2,\ldots,u_s$. Then by induction $[u_1],\ldots,[u_s]$ are all in the $(\ell-1)$-subideal generated by $x$, hence, $c$ is an element of of the $\ell$-subideal $\id_L^\ell x$ generated by $x$. 
\endproof

\begin{Lemma}\label{w1}  Let $L=L(X)$ be a free Lie algebra with the graded free basis $X$ satisfying Conditions $\mathrm{G}$ and $\mathrm{W}_m$.
Let $c$ be a nonzero linear combination of several free generators of  the same degree in $L$. Then for any $\ell\ge 1$, there are $C,\varepsilon>0$ such that the cogrowth function of the $\ell$-subideal $H$ generated by $c$ in $L$  does not exceed $C(m-\varepsilon)^n$.
\end{Lemma}

\proof We will use the fact that given a linear order on $X$, a linear basis of $L_n$ is given by LS-commutators of degree $n$. Hence the linear basis of $L$ mod $H$ can be chosen as a subset of the set of LS-commutators.

Let us first consider the case where $c=y$, an element of $X$. We can order $X$ so that $y$ is the greatest element with respect to this order. According to Lemma  \ref{ldegsub}, then any LS-commutator whose associative support contains a subword $y^\ell$ is an element of $H$. According to Lemma \ref{bezslov}, there are positive $C,\ve$ such that the number of associative words without a subword $y^\ell$  is bounded from above by $C(m-\ve)^n$, for any $n$. Thus $d_{L/H}(n)<C(m-\ve)^n$. Since $g_{L/H}(n)=d_{L/H}(1)+\cdots+d_{L/H}(n)$, we immediately observe that there is $C'>0$ such that $g_{L/H}(n)\le C'(m-\ve)^n$, as needed.

Now assume $c$ is not necessarily one letter and $c$ involves a generator $y$ with a nonzero coefficient. We have $\deg (c)= \deg (y)$. According to \cite[Lemma 2.4.1]{B}, the set $X'=(X\backslash \{ y\})\cup \{ c\}$ is another set of free generators of $L$. Notice that the degrees of elements of $L$ with respect to $X'$ remain to be the same as with respect to $X$. By the previous paragraph, $d_{L/H}(n)<C(m-\ve)^n$, and the proof is complete.
\endproof

\begin{Lemma}\label{w}  Let $L=L(X)$ be a free Lie algebra with the graded free basis $X$ satisfying Conditions $\mathrm{G}$ and $\mathrm{W}_m$.
Let $w\in L$ be any non-zero homogeneous element. Then for any $\ell\ge 1$, there are $C,\varepsilon>0$ such that the cogrowth function of the $\ell$-subideal $H$ generated by $w$ in $L$  does not exceed $C(m-\varepsilon)^n$.
\end{Lemma}

\proof  Let us consider a linear order on $X$ such that $x<y$ if $\deg(x)<\deg(y).$
We choose the minimal generator $x\in X$. The minimal ideal $L'$ of $L$ containing
all other generators  has codimension $1$ in $L. $ By Lemma \ref{pnC1}, the Lie algebra $L'$ has a free basis $X'$ consisting of all 
commutators $z$ of the form $[y,x,...,x]$, where $y\in X\backslash\{x\}.$  The free basis $X'$ of $L'$
is a graded subset of $L$  and therefore the degree of every element of $L'$
with respect to $X'$ is equal to its degree with respect to $X$. 

Note that also $X'$ satisfies Condition $\mathrm{W}_m$. Indeed, assume $\deg (x)=s$. Removing $x$ from $X$ we decrease the value $F(0)$ by $\dfrac{1}{m^s}$. Then adding all the commutators $[y,x,...,x]$ to
$X\backslash\{x\}$ we multiply the contribution of every such $y\in X\backslash\{x\}$ by $\left(1+\dfrac{1}{m^s}+\dfrac{1}{m^{2s}}+\cdots\right)$. Hence the
$F^{X'}_m(0)$ does not exceed
\begin{displaymath}
\left(1-\dfrac{1}{m^s}\right)\left(1+\dfrac{1}{m^s}+\dfrac{1}{m^{2s}}+\cdots\right)=1.
\end{displaymath}
Also the function $F^{X'}_m(\zeta)$ for $X'$ remains defined in a neighbourhood of $0$. Indeed, we obtain this new function $F^{X'}_m(\zeta)$ by first subtracting $\dfrac{1}{(m-\zeta)^s}$ from the old function $F^X_m(\zeta)$. Then we multiply the difference obtained by the function  \begin{displaymath}
(1+\dfrac{1}{(m-\zeta)^s}+\dfrac{1}{(m-\zeta)^{2s}}+\dots)=\left(1-\dfrac{1}{(m-\zeta)^s}\right)^{-1},
\end{displaymath}
which is defined in the neighborhood of zero. Thus all components of Condition $\mathrm{W}_m$ remain in place.

Because we remove $x$ of the smallest degree in $X$, also Conditions G still holds for $X'$.

We will repeat the above construction applying it to $L'$ and $X'$, and so on. After $t$ steps we obtain the Lie algebra $L^{(t)}$ and its free basis $X^{(t)}$.

Now let us show that $\displaystyle\bigcap_{t=1}^{\infty} L^{(t)}=\{ 0\}$. Indeed, by our construction, the number of generators of degree 1 in $X'$ is strictly less (if any) than in $X$ (at the same time it is possible that we get more generators of degree 2). After several steps of application of our construction, there are no more generators of degree 1. Then, in the same manner, we get rid of all generators of degree 2, and so on.  As a result, if $u$ is an element of degree $k$ and on the $t^{\mathrm{th}}$ step we got rid of all generators of degree less than or equal to $k$, we have $u\not\in L^{(t)}$.

So let us assume that $w\in L^{(t)}\setminus L^{(t+1)}$. We want to show that $w$ is a linear combination of some elements of $X^{(t)}$. Suppose to the contrary, that the expression for $w$ in $X^{(t)}$ nontrivially includes some commutators of length $n$ where $n>1$. Because $w$ is homogeneous, it cannot include commutators of length 1 equal to the generators of the minimal degree, and so $w\in L^{(t+1)}$ by definition of the transition $L^{(t)}\to L^{(t+1)}$. This contradiction proves that a finite sequence of transitions described above will brings us to the case where $w$ is a linear combination of the free generators of some $L^{(t)}$. This case has been dealt with in Lemma \ref{w1}.

We have that the $\ell$-subideal $H'$ generated by $w$ in $L^{(t)}$
has codimension growth in $L^{(t)}$ with respect to $X^{(t)}$ 
at most $C(m-\varepsilon)^n$. The same is true with respect to $X$, since the degrees of
the elements in $L^{(t)}$ with respect to these bases are equal. Since the codimension of $L^{(t)}$ in $L$ is $t$,  the
codimension growth of $H'$ with respect to $X$ in the entire $L$ is bounded by  $t+C(m-\varepsilon)^n$, where $t$ does not depend on $n$. Since $H'\subset H$, we have the same upper bound for the cogrowth
of $H$, and altering $C$ we complete the induction.
\endproof

We are now in position to complete the proof of one of the main results of this paper, Theorem \ref{t2I}.

\proof 
Every nonzero subideal $S$ contains, for some $\ell\ge 1$, an $\ell$-subideal generated by one nonzero element $w$. If we prove our claim for such subideals $S$, we will be finished.  Since we already have this proven in Lemma \ref{w} for the case where $w$ is homogeneous (in that case, with respect to the standard grading associated with the finite free basis of $L$), using Part (b) of Lemma \ref{lGR} it is sufficient to show the following. 

Let $w $ be a nonzero element of subalgebra $H$ and $u$ its leading part. Then for $\ell \ge 1 $ the $\ell$-subideal $K_\ell$ generated by $u $ in $K= \gr H $ is contained in $\gr H_\ell$, where $H_\ell$ is an $\ell $-subideal generated by $w$ in $H$. 

Indeed, by definition of $K_1$, any nonzero homogeneous element $v\in K_1$ is a linear combination of nonzero commutators  $[u,u_1,\ldots,u_s]$, where all $u_i$ are homogeneous elements in $K$. In this case each $u_i$ is a leading part of an element $w_i\in H$. Therefore, $v$ itself is a leading part of a respective linear combination of commutators  $[w,w_1,\dots,w_s]\in H_1$, so that $K_1\subset \gr H_1$.

If $\ell > 1 $ then we have $K_{\ell -1}\le \gr H_{\ell -1} $, by induction. Since $K_\ell$ is an ideal in $K_{\ell-1}$ generated by the element  $u$, it is contained in the ideal $M$ generated by $u$ in $\gr H_{\ell -1}$. But $H_\ell $ is an ideal generated by  $w$ in $H_{\ell -1}$. Hence, if we apply what we have proved for the ideals (that is, when $\ell =1$), we obtain $M\subset \gr H_\ell$. As a result, $K_l\subset \gr H_\ell$, as needed.
\endproof

\section{An example}\label{sE}

In this section, based on Lemmas \ref{ldegsub} and \ref{lPC}, we find the cogrowth function for the $2$-subideal $\mathrm{id}_L^2 x$ generated by $x$ in the free Lie algebra $L=L(x,y)$. We will need the following converse to Lemma \ref{ldegsub} in our particular case.

\begin{Lemma}\label{lPC} Let $L=L(x,y)$. Then the set $T$ of LS-commutators, different from $x$, whose associative supports do not have $x^2$ as a subword, are linearly independent modulo $H=\id_L^2 x$.
\end{Lemma}

\begin{proof} Any LS-commutator without $x^2$ in the support (excluding $x$ and $y$) is an element of the subalgebra $M$ of $L$ with free basis $\{ [x,y],[x,y,y],\ldots\}$. At the same time, $M$ is a free factor in the ideal $N=\langle M,x\rangle=\id^1_L x$. Both claims follow by Lemma \ref{pnC1}. Therefore, $M$ trivially intersects the ideal generated by $x$ in $N$ which is exactly $H=\id^2_L x$. It then follows that all the set $T$ of commutators in question, excluding $y$, being linearly  independent in $L$ and $M$, remains linearly independent modulo $\id^2_L x$.  Restoring $y$ to the set does not change the situation because $y\notin \langle M,x\rangle$, while all the rest (including $\id^2_L x$) is a subset of $\langle M,x\rangle$.
\end{proof}

We will use a linear basis of $L$ composed of LS-commutators. It follows from Lemmas \ref{ldegsub} and \ref{lPC}, that we need to find the growth of the set of associative  LS-words different from $x$, that do not have $x^2$ as a subword.

We first count the number of words of length $n$ each being the product of the subwords $y$ and $yx$. Let $a_n$ denote the number of such words with suffix $x$ and $b_n$ with suffix $y$. Then $a_n =a_{n-2}+b_{n-2}$ (these are obtained by attaching $yx$ to the words of length $n-2$) and $b_n=a_{n-1}+b_{n-1}$. Then also $b_{n-1}=a_{n-2}+b_{n-2}=a_n$. It follows that $b_n$ is the $n\th$ term of the Fibonacci sequence $1,1,2,3,5,\ldots$,  while $a_n$ is the $(n-1)\st$ term of the same sequence (we assume the zeroth term being $0$).

The total number of these words of length $n$ is now $a_n+b_n=b_{n-1}+b_n$, hence the $(n+1)\st$ term of the Fibonacci sequence. Actually, we have counted the number of words $w$ of degree $n$, starting with $y$, such that no cyclic shift of $w$ contains $x^2$, as a subword. We will say that $w$ does not contain a \textit{cyclic subword} $x^2$.

The set of words we have just counted is the set of all words of length $n$ without subword $x^2$, with the exclusion of those words that have prefix $x$. It is obvious that the number of words of length $n$ without cyclic subword $x^2$ with prefix $x$ is equal to the number of words without cyclic subword $x^2$ with suffix $x$. Each of these latter can already be written via $y$ and $yx$ and has suffix $x$. Earlier, we denoted the number of such words as $a_n$. As a result, we have the total number $c_n$ of the words of length $n$ without cyclic subwords $x^2$ equal to $(a_n+b_n)+a_n$ (with the exception of $x$). Thus, $c_n$ is the sum of the $(n+1)\st$ and $(n-1)\st$ terms of the Fibonacci sequence.

Now we want to delete from the above number those subwords that are powers of shorter words. Notice that every proper power is a proper prime power. Note that if $v$ is a word without cyclic subwords $x^2$ then each power of $v$ enjoys the same property. Now the computation of the number of words which are prime powers is routine. If $p|n$, where $p$ is prime, then $c_{n/p}$ of our words will be $p\th$ proper powers. These must be subtracted, for each prime divisor $p$ of $n$. 

If $d=p_1p_2$ is the product of two different primes, then the $d\th$ powers have appeared in the above process twice, so while considering such powers we have to correct this error, etc. All this leads to a formula for the number $d_n$ of words of length $n$ without cyclic subwords $x^2$, where the M\"obius function plays the role of the ``corrector'': $\displaystyle d_n =\sum_{d|n} \mu(d)c_{n/d}$. Although we considered only those divisors $d$ which split as the product of pairwise different prime divisors, we have included in the preceding formula all divisors $d\,|\,n$ simply because $\mu(d)=0$ as soon as $d$ is divisible by a square of a prime number.

The set of all words we just counted splits into $n$-element subsets consisting of cyclic shifts of one of its elements. The greatest word in this subset, in the sense of the LEX-grading such that $x>y$,  is an LS-word (see Subsection \ref{ssRLFB}). As a result, the number of all LS-words of length $n$ in the alphabet $x,y$, without cyclic subwords $x^2$ is now given by the formula $ (1/n)\sum_{d|n} \mu(d)c_{n/d}$ (as earlier, $x$ is not counted here). Since no LS-word of length $\ge 2$ can have suffix $x$, this formula also gives the total number of LS-words of length $n$, without subword $x^2$.

\begin{Proposition}
Let $L=L(X)$ be a free Lie algebra of rank 2, with free basis $X$, $x$ a letter in $X$, $S$ a $2$-subideal of $L$ generated by $x$. Let $\fib(n)$ denote the $n\th$ Fibonacci number. Then the graded function $d_{L/S}$ of the relative cogrowth of $S$ in $L$ is given by
\begin{equation}\label{eFib}
d_{L/S}(n)=\frac{1}{n}\sum_{d|n} \mu(d)\left(\fib\left(\frac{n}{d}-1\right)+\fib\left(\frac{n}{d}+1\right)\right).
\end{equation}
\end{Proposition}

We can see from (\ref{eFib}) that the growth of the functions $d_{L/S}$ and $g_{L/S}$ is exponential, with the base of exponent $\displaystyle\lim \sqrt[n]{g_{L/S}(n)}$ being the same number as for the Fibonacci sequence, that is, $\dfrac{1+\sqrt{5}}{2}$. Recall (see (\ref{en1}))  that for the whole of $L$ the base of exponent is $2$.

The first values of the functions $d_{L/S}$ and $g_{L/S}$ are given in the Table 1 below.

\nopagebreak

\begin{table}[ht]

\centering

\begin{tabular}{|c| c| c| c| c| c| c| c| c| c| c|}

\hline
\hline 
$n$ & 1 & 2 & 3& 4 & 5&6&7&8&9&10 
\\
\hline
\hline
$d_{L/S}(n)$&1 & 1 & 1 & 1 & 2&2&4&5&8&11\\
$g_{L/S}(n)$&1 & 2 & 3 & 4 & 6&8&12&17&25&36\\
\hline
\hline
$n$ &11&12&13&14&15&16&17&18&19&20\\
\hline
\hline
$d_{L/S}(n)$&18&25&40&58&90&135&210&316&492&750\\
$g_{L/S}(n)$&54&79&119&177&267&402&612&928&1420&2170\\
\hline
\hline
\end{tabular}
\caption{Cogrowth of the 2-subideal generated by $x$ in $L(x,y)$}
\label{table:Example}
\end{table}

\section{Computing the exponential base of the relative growth}\label{sCBE}

The approach suggested in Section \ref{sCGW} allows us to prove our claims about the exponential base of  of the growth of (not necessarily finitely generated) subalgebras in free Lie algebras, as stated in Theorem \ref{t1aI}. To be able to apply the results of Section \ref{sCGW}, we need an auxiliary lemma. 

\begin{Lemma}\label{lAtL} Let $A$ and $L$ be a free associative algebra and a free Lie algebra, with the same graded free basis $Y$. Suppose  $\displaystyle\lim_{n\to\infty}\sqrt[n]{g_A(n)}$ exists and $>1$. Then $\displaystyle\beta_L=\lim_{n\to\infty} \sqrt[n]{g_L(n)}$ exists and is equal to $\displaystyle\lim_{n\to\infty}\sqrt[n]{g_A(n)}$.
\end{Lemma}
\begin{proof}  
If $d|n$ then the product of any $d$ words of degree $n/d$ is a word of degree $n$, which implies $d_A(n/d)^d\le d_A(n)$. If additionally, $d\ge 2$ then $d_A(n/d)\le \sqrt{d_A(n)}\le \sqrt{g_A(n)}$. But $d_A(n/d)$ is the number of words of degree $n$ which are $d$-powers. After summation over all divisors $d\ge 2$, we will find that the number of words of degree $n$ does not exceed $n\sqrt{g_A(n)}$. 

Now the existence of $\displaystyle\lim_{n\to\infty}\sqrt[n]{g_A(n)}>1$ means that $\sqrt{g_A(m)}<g_A(n)^{2/3}$, for all sufficiently great $n$ and all $m\le n$. Therefore, after summation over $m\le n$, we will have for large enough $n$ that the number of proper powers among the words of degree at most $n$ is bounded from above by $n^2 g_A(n)^{2/3}$, which is $o(g_A(n))$, because $g_A(n)$ is exponential. It follows that the number of words of degree $\le n$, which are not proper powers is $g_A(n)(1-o(1))$. 

Since among the cyclic shifts of any such word of degree $m\le n$ precisely one is an LS-word, the number of these latter takes the form of $g_A(n)(1-o(1))\gamma(n)$, where $\gamma(n)\in [1/n, 1]$.  By Lemma \ref{lnS}, the number of LS-words of length at most $n$ is $g_L(n)$. Taking the $n\th$ roots of these numbers, we arrive at the following:
 \begin{displaymath}
 \displaystyle\lim_{n\to\infty}\sqrt[n]{g_L(n)} =\displaystyle\lim_{n\to\infty}\sqrt[n]{g_A(n)},
\end{displaymath}
as claimed.
\end{proof}

Now we can proceed to the proof of Theorem \ref{t1aI}.

\begin{proof}
Recall that we are dealing with the relative growth a subalgebra $H$  of a free Lie algebra  $L$ of rank $m$ with a nongraded free basis $X$. Using Lemmas \ref{lGR} and \ref{pnC2} allows us to always assume that $H$ is a homogeneous subalgebra generated by an irreducible free basis $Y$, in which the number of elements of degree $i$ equals $k_i$, for $i=1,2\ld$. The relative growth of $H$ in $L$ is the same as the absolute growth of a free Lie algebra $L(Y)$, where $Y$ is viewed as an abstract graded set. The degree of each $y\in Y$ equals the degree of $y$ with respect to $X$ in $L$. Let us embed $L(Y)$ in a free associative algebra $A(Y)$. Lemma \ref{lAtL} allows us to translate our claims about $H$ to the claims about $A(Y)$.

We start with Claim (i), where $H$ is not necessarily finitely generated but is nonabelian.  Let us use Lemma \ref{ldXn} (notice that in that Lemma the graded set of generators is denoted by $X$!) Since $\dim L_i\le m^i$, it follows that $k_i\le m^i$, for each $i=1,2,\ld$. It follows that for any  $z>m$ the series $F(0)$, as a function of $z$, converges, hence $z_0$ in formula (\ref{edXn}) exists and provides us with the exponential base for the growth of  $A(Y)$, hence for the growth of $L(Y)$ (Lemma \ref{lAtL}), hence for the relative growth of $H$ in $L$. As noted in Lemma \ref{ldXn}, if $\# Y>1$, the growth is exponential. Thus we have proved Claim (i).

In the proof of Claim (ii),  the function $F(\zeta)=F_z^Y(\zeta)$ takes the form of 
\begin{equation}\label{pbaseexp}
F(\zeta)=\frac{k_1}{z-\zeta}+\frac{k_2}{(z-\zeta)^2}+\cdots+\frac{k_d}{(z-\zeta)^d}.
\end{equation}
where $d$ is the maximal degree of the elements in $Y$, and each $k_i$ is the number of elements of degree $i$, $i=1,2,\ldots,d$. As in the proof of Claim (i), we use  Lemma \ref{ldXn}. To be able to apply formula (\ref{edXn}), we need to determine zeros of $F(0)-1$, as a function of $z$. One can write $F(0)$, as a function of the real argument $z>0$ in the form $\dfrac{f(z)}{z^d}$, where $f(z)$ is a polynomial of degree less than $d$, with nonnegative coefficients and positive free term $k_d$. Using the Descartes Rule of Signs, the equation $z^d-f(z)=0$ has a unique positive root $z_0$.  Since $z_0$ is a root of a monic polynomial with integral coefficients, our claim about the algebraic integrality of the exponential base of $H$ follows. The exponential base is less than $m$, which follows from the exponential negligibility of  Theorem \ref{t1I}.

We prove (iii) for the case $m=2$, the general case being quite analogous. Choose $m_0\in (1,2)$. Let $M$ be the subalgebra of codimension 1, with free basis $u_1=x$, $u_2=[x,y],\ld$, given in Lemma \ref{pnC1}. For this subalgebra, $\displaystyle\sum_{i=1}^\infty \frac{1}{2^i}=1$. Suppose we could discard few summands in the series  so that 
\begin{equation}\label{eM0}
\sum_{i=1}^\infty\frac{k_i}{m_0^i} =1,
\end{equation} 
where $k_i=0$ if $u_i$ is discarded and $k_i=1$ otherwise. Let $B$ be the subalgebra of $M$ generated by those $u_i$ for which $k_i=1$. Then, according to our previous argument, $m_0$ will be the exponential base for the  $B$.

To find the sequence $k_1,k_2,\ld$ so that (\ref{eM0}) holds, let us assume by induction, that we have chosen $k_1,\ld,k_j$ so that $\displaystyle 1-\sum _{i=1}^j\frac{k_i}{m_0^i}\in \left[0,\frac{1}{m_0^j}\right)$. Let us set  $\displaystyle a_j=1-\sum _{i=1}^j \frac{k_i}{m_0^i}$. If $\displaystyle a_j< \frac{1}{m_0^{j+1}}$, then we set $k_{j+1}=0$, otherwise, $k_{j+1}=1$. Since $m_0<2$, in either case,  $a_{j+1}$ is in the desired interval. Since $m_0>1$, it follows that $a_j\to 0$. Hence $\displaystyle\sum_{i=1}^\infty \frac{k_i}{m_0^i}=1$, proving Claim (iii).

Claim (iv) is an easy consequence of the construction in the proof of Claim (iii). Indeed, the subalgebra $H$ in that proof is generated by a subset of the set of  the canonical free basis of $M$. Let $S$ be the ideal of $M$ generated by the complement of that set. We have a vector space decomposition $L=H\oplus S\oplus \langle y\rangle$, where all subspaces are homogeneous. In this case, for any $n\ge 1$, $L^{(n)}=(L^{(n)}\cap H)\oplus (L^{(n)}\cap S)\oplus \langle y\rangle$. Hence,
\begin{displaymath}
g_{L/S}(n)=\dim L^{(n)}/ (L^{(n)}\cap S) = \dim (L^{(n)}\cap H)+1=g_H(n)+1.
\end{displaymath}
It follows that the exponential base of $g_{L/S}$ is he same as that for $g_H$. Finally, $M/S\cong H$, a free Lie algebra.
\end{proof}

We conclude this section by exhibiting a positive algebraic integer which is not the exponential base of any finitely generated subalgebra. For instance, if $H$ is generated in $L(x,y)$ by $x,[x,y]$ then by Claim (ii) in Theorem \ref{t1aI}, we have to solve $1+z=z^2$. As a result, we obtain $z_0=\dfrac{1+\sqrt5}{2}$. As a side remark, not every positive algebraic integer $<m$ is the exponential base for the growth of a finitely generated subalgebra in a free Lie algebra of rank $m$. For instance, $\lambda =(5-\sqrt 5)/2$, which is a positive root of $z^2-5z+5$, cannot serve as such a base. Indeed, if $\lambda$ is the solution to the equation with integral coefficients $f(z)=z^d$ then another root $\mu =(5+\sqrt 5)/2>0$ of $z^2-5z+5$ is also a solution. At the same time, by the argument in the proof of Claim (ii) the positive solution $z_0$ of $f(z)=z^d$ must be unique!

\section{Free complements}\label{sFC}

The main contents of this section is the proof of Theorem \ref{t4I}.

\proof
We first consider the case where $H$ is a homogeneous subalgebra in $L$ with a homogeneous set $B_0$ of free generators. Since $B_0$ is finite, there is $t$ equal to the maximum of degrees of all elements in $B_0$. For the homogeneous component $L_t$ of $L$ we have $L_t=(H\cap L_t)\oplus M_t$, for some (homogeneous) complementary subspace $M_t$. Let us denote by $B_1$ a linear basis of  $ M_t$ and consider the subalgebra  $H_1$ generated by $H$ and $B_1$. We want to prove  that $B(1)=B_0\sqcup B_1$ is a free basis for $H_1$. 

For this, it is necessary and sufficient (see \cite[Theorem 4.2.11]{B}) to show that $B(1)$ generates $H_1$ (this is clear) and that the elements of $B(1)$ are linearly independent modulo the commutator subalgebra $[H_1,H_1]$. Let us denote by $L^m$ the $m\th$ term of the lower central series for $L$. We have $L^m=L_m\oplus L_{m+1}\oplus\cdots$. Since $[M_t,H]\subset L^{t+1}$, we have that  $[H_1,H_1]\subset [H,H]\oplus L^{t+1}$. So a nontrivial linear dependence of the elements of  $B(1)$ modulo $[H_1,H_1]$ would imply a nontrivial linear dependence of the elements of $B(1)$ modulo $[H,H]+ L^{t+1}$. The subspaces $[H,H]$ and $L^{t+1}$ are homogeneous. It follows then that we can consider the linear dependence only for  homogeneous elements of $B(1)$ of the same degree $d$. 

If $d<t$ then neither $L^{t+1}$ nor $B_1$ have homogeneous  elements of degree $d$. Thus we have a nontrivial linear dependence of some elements of $B_0$ modulo $[H,H]$. This is not possible because $B_0$ is a free basis of $H$. The case $d>t$ is vacuous because neither $B_0$ nor $B_1$ have elements of such degree. Now if $d=t$ then neither $B_0$ nor $L^{t+1}$ have elements of degree $t$. As a result, we have a nontrivial linear dependence between the elements of $B_1$ modulo $[H,H]\cap L_t\subset H\cap L_t$. This is not possible by the choice of $B_1$ as a basis of a direct complement to $H\cap L_t$ in $L_t$. Thus we have proved that $B(1)$ is indeed a free basis for $H_1$.

Since $H_1$ is homogeneous and the degrees of elements in $B(1)$ are less than $t+1$, we can repeat the previous construction by replacing $t$ with $t+1$. So we write  $L_{t+1}=(H\cap L_{t+1})\oplus M_{t+1}$, choose a linear basis $B_2$ in $M_{t+1}$, set $ B(2)=B_0\sqcup B_1\sqcup B_2$, introduce the subalgebra $H_2$ generated by $B(2)$, prove that $B(2)$ is a free basis of $H_2$, etc. The union $C=\bigcup_{i=1}^\infty  H_i$ is a free Lie subalgebra with basis $B=\bigsqcup_{i=0}^\infty B_i$. Clearly, $H$ is a free factor in $C$.

By construction, any homogeneous polynomial of degree $\ge t$ is an element of $C$. It follows that the codimension of $C$ in $L$ is at most the dimension of the $t\th$ term $L^{(t)}$ of the degree filtration of $L$. Since $L$ is finitely generated, we  have $\dim L^{(t)}\le\infty$, proving that $C$ is of finite codimension in $L$.

Now  suppose that $H$ is not necessarily homogeneous. Recalling Section \ref{sCGFGS}, we choose a finite irreducible set $S$ of free generators for $H$ and consider the set $S'$ of leading parts of the elements of this set. Let $H'$ be the subalgebra generated by $S'$. By Lemma \ref{pnC0}, $S '$ is independent. Let us set $B_0=S'$ and proceed as just above to produce a homogeneous subalgebra $C'$ of finite codimension in $L$ with the set $B'$ of free generators such that $B'=B_0\bigsqcup \overline{B_0}$, where $\overline{B_0}=\bigsqcup_{i=1}^\infty B_i$. Now let us consider $B=S\bigsqcup\overline{B_0}$. Then the set of the leading parts of elements in $B$ is $B'$. Since $B'$ is independent, by Lemma \ref{pnC0}, the same is true for $B$.

Hence $B$ is the free basis of a subalgebra $C$. Invoking Lemma \ref{pnC2}, we have that $C'=\gr C$. If  $K$ is a subalgebra generated by $\overline{B_0}$, we have $C=H\ast K$. Notice that by construction the subalgebra $K$ is homogeneous.

It remains to note that $\dim L/C<\infty$, which follows by Part (b) of Lemma \ref{lGR}. Now the proof is complete.
\endproof

It is worth noting that in a particular case where the subalgebra $H$ of a free Lie algebra $L$ is generated by its homogeneous component of degree $t$, it is true that $H$ is a free factor in the $t^\mathrm {th} $ term $L^t$ of the lower central series of $L$. Probably, this result was already known to A.I.Shirshov.

\section{Triviality of ideals in subideals}\label{sTIS}

We start this section with a fairly general result about the derivations of free algebras, which we think could be of interest in its own.

\begin{Lemma}\label{lInvDer} Let $A$ \emph{(}respectively, $L$\emph{)} be the free associative \emph{(}respectively, Lie\emph{)} algebra with free generators $x_1,x_2,\ldots$ over an arbitrary field and $D$ its derivation given on the generators by the rule $Dx_i=x_{i+1}$ \emph{(}$i=1,2,\dots$\emph{)}. Suppose $a$ is a
nonzero element  of $A$ \emph{(}of $L$\emph{)}. Then for every $k\ge 1$, there exists
$n\ge 0$ such that the element $D^n(a)$ does not belong to the ideal $I_k$
\emph{(}to the Lie ideal $J_k$\emph{)} of $A$ \emph{(}resp., of $L$\emph{)} generated by the set $\{x_1,\dots,x_k\}$.
\end{Lemma}

\proof Let us begin with the associative case. We may assume that $a$ is a homogeneous
element of degree $c\ge 1$ and $a$ is an element of a subalgebra $A_l$ generated by
$x_1,\dots,x_l$ for some $l\ge k$. We can write
\begin{displaymath}
a=\sum_{1\le i_1\le\ell,\ldots,1\le i_c\le\ell}\alpha_{i_1,\dots,i_c}x_{i_1}\cdots x_{i_c}
\end{displaymath} 
for some scalars $\alpha_{i_1,\dots,i_c}.$ 
We say that a $c$-tuple $(i_1,\dots,i_c)$ is an element of the \textit{support} $\supp{a}$ if $\alpha_{i_1,\dots,i_c}\ne 0.$

Let us define a sequence of numbers $K_1,K_2,\ldots, K_c$ by setting $ K_d = (1/2)(2k+2)^{2^{c-d}}$, for any of $d=1,2,\ldots,c$. Then $K_c=k+1$ and $K_d=2K_{d+1}^2$, for any $1\le d<c$. Our lemma is an easy consequence of the following.

\begin{Claim}
For any $d$, $1\le d\le c$, there is $n\ge 1$ such that  the support of $D^n(a)$ has a $c$-tuple $(i_1,\dots,i_c)$ with $\min\{ i_1,\dots,i_d\}\ge K_d$.
\end{Claim}

The estimate for the value of $n$ can be recovered from the proof that follows. Once proven, our Lemma follows because  $K_c>k$.

For the proof, let us use induction by $d$. 

In the case $d=1$, our task is easy. Indeed, let $(i_1,...,i_c)\in\supp{a}$ with the greatest
possible $i_1$.
Then the only  $c$-tuple
$(i_1+1,i_2,\dots,i_c)$ comes to $\supp{D(a)}$ from 
\begin{displaymath}D(\alpha_{i_1,i_2,\dots,i_c}x_{i_1}x_{i_2}\cdots x_{i_c})=\alpha_{i_1,i_2,\dots,i_c}(x_{i_1+1}\cdots x_{i_c}
+x_{i_1}x_{i_2+1}\cdots x_{i_c}+\dots),\end{displaymath}
and $i_1+1$ is the greatest first index in the members of $\supp{D(a)}$. Then we repeat
the argument and have  $(i_1+K_1,i_2,\dots,i_c)\in \supp{D^n(a)}$ if $n=K_1.$

Now assume that the statement is true for some $d$ ($1\le d<c$). In this case, there is $n$ such that the support of
$b=D^n(a)$ includes a $c$-tuple $(i_1,\dots,i_c)$ with $\min\{ i_1,\dots,i_d\}\ge K_d.$
Let us split $\supp{b}$ as the disjoint union 
\begin{displaymath}\supp{b}=\left(\bigsqcup_{j\ge K_{d+1}} S_j\right)\bigsqcup S,
\end{displaymath}
 where $(i_1,\dots,i_c)\in S_j$ if $\min\{ i_1,\dots,i_d\}=j$, $j\ge K_{d+1}$. The $c$-tuples in $S$ satisfy $\min\{i_1,\dots,i_d\}<K_{d+1}$.
The maximal $j$ with  non-empty $S_{j}$ satisfies the inequality $j\ge K_d$ because of the way we have obtained $b$.

To proceed,  we introduce a function $f(j)$ ($j\ge K_{d+1}$) whose $j^{\mathrm{th}}$ value equals to the maximal value of the $(d+1)^{\mathrm{st}}$
coordinate of the tuples in the subset $S_j$; we set, by definition, $f(j)=1$ if the subset $S_j$ is empty.

If $f(j)\ge K_{d+1}$ for some $j\ge K_{d+1},$
then there is a $c$-tuple $(i_1,\dots,i_c)\in S_j$ such that $\min(i_1,\dots,i_{d+1})\ge K_{d+1},$ which is sufficient to complete
the induction step. Therefore, we further assume that $f(j)<K_{d+1}$ for every $j\ge K_{d+1}$, and so $f(j)$ takes less than  $K_{d+1}$ different values. 

It follows that the length of any decreasing sequence  $j(1)>j(2)>\dots$
such that the values $f(j(1))<f(j(2))<\dots$ increase, is less than $K_{d+1}$.
Recall that the maximal value of $j$ with non-empty $S_j$ is at least $K_d$.
Let us choose a particular decreasing sequence as follows. There exists $j(1)\ge K_d$ such that $f(j(1))\ge f(q)$ for every $q\ge j(1)$. If for all $j<j(1)$ we have $ f(j)\le f(j(1))$, then we are done. Otherwise we keep working and define $j(2)$ as the maximal
index such that $j(2)<j(1)$ while $f(j(2))>f(j(1))$. Proceeding in the same way, we eventually arrive at a finite sequence $j(1)>j(2)>\ldots>j(s)$ with $f(j(1))<f(j(2))<\ldots<f(j(s))$, which cannot be extended any further. Here we additionally have $s< K_{d+1}$ and $f(j(t))\ge f(q)$, for arbitrary $t\le s$ and $q\ge j(t)$.

Notice that it is not possible that every number in the sequence $j(1)-j(2),\dots, j(s-1)-j(s), j(s)$
is strictly less than $2K_{d+1}$. For, if this were the case, we would be able to write 
\begin{displaymath}
K_d\le j(1)=(j(1)-j(2))+\dots+(j(s-1)-j(s))+j(s)<2K_{d+1}K_{d+1}.
\end{displaymath} 
Therefore,  either $j(t)-j(t+1)>2K_{d+1}$ for some $t<s$ or $j(s)>2K_{d+1}$. In both cases, we can conclude that there is a number  $m$ such that the following are true:
 \begin{displaymath}
 m\ge 2K_{d+1},\mbox{ and } f(m)\ge f(q)\mbox{ for any }q\ge m-K_{d+1}.
 \end{displaymath}
Now we take $(i_1,\dots,i_d,f(m),\dots,i_c)\in S_m$, consider the respective monomial $u=x_{i_1}\cdots x_{i_d}x_{f(m)}\cdots x_{i_c}$,
entering $b$ with non-zero coefficient, and compute the elements $D^{K_{d+1}}(u)$ and $D^{K_{d+1}}(b)$.
The support of $D^{K_{d+1}}(u)$ includes the $c$-tuple
\begin{displaymath}
{\bf i}=(i_1,\dots,i_d,f(m)+K_{d+1},i_{d+2},\dots,i_c).
\end{displaymath} This follows
exactly as in the case $d=1$ above. The first $d+1$ indexes of this tuple are not less than $K_{d+1}$,
by the definition of $S_m$. Therefore to complete the induction, it suffices to prove that for any monomial $v$ occurring in $b$ with nonzero coefficient, the tuple $ {\bf i}$ does not appear in the support of $D^{K_{d+1}}(v)$.  Let $\supp{v}=\{(i'_1,\dots,i'_d,\dots,i'_c)\}$.

\begin{enumerate}
\item[{\bf Case 1}]:  $\supp{v}\in S_q$, where $q\ge m-K_{d+1}$.  By the choice of $m$, $f(m)\ge f(q)$.
Our argument, as just above, shows that $D^{K_{d+1}}(v)$ has at most one monomial with multi-index of the form 
\begin{displaymath}
{\bf i'}=(i'_1,\dots,i'_d,f(m)+K_{d+1},\dots,i'_c).
\end{displaymath} 
Since ${\bf i'\ne i}$ we have ${\bf i}\notin\supp{D^{K_{d+1}}(v)}$.
\item[{\bf Case 2}]: $\supp{v}\in S_q$,  where $q<m-K_{d+1}.$ According to the definitions of $S_q$ and $S_m$, we have $\min(i'_1,\dots,i'_d)=q<m-K_{d+1}$. Each  of the $K_{d+1}$ derivations can increase the minimum of the first
$d$ indices in the members of the support at most by $1$.  Hence ${\bf i}\notin \supp{D^{K_{d+1}}(v)}$ .
\item[{\bf Case 3}]:\textit{ $\supp{v}\in S$.} Then by the definitions of $D$ and $S$, one of the first $d$ indices of any term of $D^{K_{d+1}}(v)$
is less than $2K_{d+1}$ .  Hence, by the definition of $m$ and $S_m $, ${\bf i}\notin\supp{D^{K_{d+1}}(v)}$.
\end{enumerate}

Thus the proof is complete in the case of associative algebras. The case of Lie algebras reduces to the associative case because the free generators of $A$ also freely generate $L$ with respect to the bracket operation, $J_k\subset I_k$, and the derivation
$D$ of $L$ extends to a derivation of $A$ (see \cite[Chapter II]{NB}).

\endproof

\begin{Theorem}\label{t3} Let $J$ be a proper ideal of a free Lie algebra $L$ and $S$ a finite subset of $J$. Let $I=\mathrm{id}_J S$ be the ideal closure  of $S$ in $J$. Then for every $z \in L\backslash J$ and for every nonzero $a\in I$ there is  $n$ such that $[a,z,...,z]$ ($n$ times) is not an element of $I$. 
\end{Theorem}

\proof One may assume that $L=\langle J, z\rangle$ where $J$ is an ideal of codimension  $1$ in $L$. Then an element $z'=z-v$, where $v\in J$ can be included in a free basis of  $L$.  If for any number $n$ of occurrences of $z'$ we had $b_n=[a,z',...,z']\in I$ then we would also have that all  $a_n=[a,z,...,z]\in I$. To prove this, we first note that, if we replace $z$ by $z'+v$ then we will have the sum of left-normed commutators $w=[a,u_1,u_2,\ldots,u_n]$,  where each $u_i$ is either $z'$ or $v$. Using the Jacobi identity $[v,s,t]=[v,t,s]+[v,[s,t]]$, we can change places $z'$ and $v$ in $w$, if $z=t$ follows $v=s$ in $w$. The resulting commutators would still be of the form of $w$ but now $u_i$ are either $z'$ or some elements of $J$. If there is still some $z'$ to the right of some $u_i$, we repeat the process. When all $z'$ are to the left of all $u_i$ we have the elements of the form $[a,z',\ldots,z',u_1,\ldots,u_k]$. The initial portion $[a,z',\ldots,z']$ is in $I$, by our assumption. Since $I$ an ideal of $J$, the whole commutator is in $I$, as well. 

Hence, we may assume that $z$ is an element of the free basis of $L$.  Then the inner derivation $\ad z$ shifts the free generators of $J$ (see Lemma \ref{pnC1}), so we are tempted to apply Lemma  \ref{lInvDer}. However, the direct application is not possible, because there can be many different ``orbits'' for the action of $\ad z$ on the free basis of  $J$. Hence, we need an easy modification of Lemma  \ref{lInvDer}, as follows.

\begin{Lemma}\label{lInvDerMO} Let $A$ \emph{(}respectively, $L$\emph{)} be the free associative \emph{(}respectively, Lie\emph{)} algebra with the set $X$ of free generators split as a disjoint union of subsets $X^{(j)}=\{ x_1^{(j)},x_2^{(j)},\ldots\}$ over arbitrary field and $\delta$ its derivation given on the generators by the rule $\delta x_i^{(j)}=x_{i+1}^{(j)}$ \emph{(}$i=1,2,\dots$\emph{)}. Suppose $a$ is a
nonzero element  of $A$ \emph{(}of $L$\emph{)}. Then for every finite subset $S$ of $X$, there exists
$n\ge 0$ such that the element $\delta^n(a)$ does not belong to the ideal $I_S$
\emph{(}to the Lie ideal $J_S$\emph{)} of $A$ \emph{(}resp., of $L$\emph{)} generated by $S$.
\end{Lemma}

\proof This is a direct consequence of Lemma \ref{lInvDer}. Indeed, considering, say, the associative case, let us assume that, for any $n$, we have $\delta^n(a)\in I_S$. Let us consider a homomorphism of associative algebras $\vp$ from $A(X)$ to the free associative algebra $A(x_1,x_2,\ldots)$ of Lemma \ref{lInvDer}, defined in the following way. Suppose an expression for $a$ in $A(X)$ includes only at most $m$ first letters from each of subalphabets $X^{(j)}$, $j=1,2,\ldots$. Then we set $ \vp(x_i^{(1)})=x_i$, $ \vp(x_i^{(2)})=x_{i+m}$, $ \vp(x_i^{(3)})=x_{i+2m}$, etc., where $i=1,2,\ldots$ Clearly, under this homomorphism, $\vp(a)\neq 0$ because we simply applied a bijective change of the sets of free variable involved in $a$. Now both algebras $A(X)$ and $A(x_1,x_2,\ldots)$ are equipped with derivations, $\delta$ in the former case and $D$, as in Lemma \ref{lInvDer}, in the latter, that is, $D(x_i)=x_{i+1}$, $i=1,2,\ldots$. Clearly, $\vp$ is a \textit{differential homomorphism}, that is $\vp\circ\delta=D\circ\vp$.

Since we assumed that $\delta^n(a)\in I_S$, for all $n$, applying $\vp$ would mean that $D^n(\vp(a))\in I_{\vp(S)}$. This comes as contradiction to Lemma \ref{lInvDer}.
\endproof

Theorem \ref{t3} follows from Lemma \ref{lInvDerMO}.
\endproof

Clearly, Theorem \ref{t0601} from the Introduction is the direct consequence of the just proven Theorem \ref{t3}.

Now we came closer to the proof of Corollary \ref{t3I1} from the Introduction. 

\proof  Let $J(i)$ be the $i\th$ subideal closure of a finite set $S$ of elements of $J=J(0)$,  as in Corollary \ref{t3I1}, and  $L\supset M_1\supset \ldots$ an arbitrary subideal series in $L$ with nonzero terms. We need to show that $M_\ell $ is not contained in $J(\ell)$.

Let us first prove that setting $N_i=J(i-1)\cap M_i$, $i=1,2,\ldots$, the series $L\supset N_1\supset N_2\ldots $ is also a subideal series with nonzero terms.  First of all, if $P\cap Q=\{ 0\}$ for $\{ 0\}\neq P,Q\triangleleft L$, $a\in P$, $b\in Q$ then $[a,b]=0$. Now the subalgebra $M$ generated by $a,b$ is free; since $a,b$ commute, $M$ must be abelian, hence $\dim M=1$. So $a,b$ are scalar multiples of each other, a contradiction with $P\cap Q=\{ 0\}$. Therefore, the intersection $N_1$ of two nonzero ideals $M_1$ and $J=J(0)$ of a free Lie algebra $L$ is again nonzero. Now since $N_1$ and $J(1)$ are two nonzero ideals of $J(0)$, their intersection $N_1\cap J(1)$ is nonzero. Similarly, $M_2\cap N_1$ is nonzero. Hence $(M_2\cap N_1)\cap (N_1\cap J(1))$ is nonzero, being the intersection of two nonzero ideals of $N_1$. As a result, $N_2=M_2\cap  J(1)\ne\{ 0\}$. Continuing in the same way, we find that all $N_i=M_i\cap J(i-1)$ are nonzero.

As a result, when we resume proving Corollary \ref{t3I1}, we may assume that $M_\ell\subset J(\ell-1)$, for all $\ell$. So we need to prove that $M_\ell\not\subset J(\ell)$, for any $\ell=1,2,\ldots$. Assume the contrary, that is, $M_\ell\subset J(\ell)$, for some $\ell=1,2,\ldots$.  If $\ell=1$, we can apply Theorem \ref {t0601}, because $J$ is a proper ideal in $L$ while $J(1)$ is an ideal generated in $J=J(0)$ by a finite subset of elements. It follows that $J(1)$ does not contain an ideal $M_1$ of the whole algebra $L$. 
Thus, if there is $\ell$ such that $M(\ell)\subset J(\ell)$ then $\ell > 1$. In the case this happens, let $\ell$ be minimal with this property. Consider the ideal closure $R$ of  $M(\ell)$ in $J(\ell-1)$.  Then $R$ is contained in $J(\ell)$, which is the ideal closure of the finite set $S\subset J(\ell-1)$. By the minimality of the choice of $\ell$, there is $z\in M_{l-1}$, which is not an element of $J(l-1)$. As established earlier, $z\in M({l-1})\subset J(l-2)$. Since $[z, M(\ell)]\subset M(\ell)$ and $[z, J(\ell-1)]\subset J(\ell-1)$, we have that $[z, R]\subset R$. This easily follows using the  Jacobi identity because every element of $R$ is a linear combination of the commutators of the form $[y,x_1,\ld,x_m]$, where $y\in M(\ell)$ and $x_1,\ld,x_m\in J(\ell-1)$. As a result, $R$ is a nonzero ideal in the subalgebra $P$ generated by $J(\ell-1)$ and $z$. This contradicts Theorem \ref{t0601} applied to the algebra $P$ and its proper ideal $J(\ell-1)$. Hence, the proof of Corollary \ref{t3I1} is complete. 
\endproof

\begin{Remark}\label{rlast} In the case of associative algebras we have a very simple example when there is an ideal $I$ generated by finitely many elements in a proper ideal $J$ of a free associative algebra $A$, and still containing a nonzero ideal of $A$. For this one can take $A$ with free basis $\{ x\}$, $J$ the principal ideal of $A$ generated by $x^2$, $S=\{ x^2, x^3\}$, $I$ the ideal of $J$ generated by $S$. The reader would easily check that the nonzero ideal of $A$ contained in $I$ is \ldots $I$ itself! 
\end{Remark}

\end{document}